\documentclass[11pt]{amsart}
\usepackage[margin=1in]{geometry}

\usepackage{graphicx, physics, mathrsfs}
\usepackage[colorlinks=true, citecolor=blue]{hyperref}

\newcommand{\subeq}{\subseteq}

\newcommand{\ol}{\overline}

\newcommand{\bigp}[1]{\left( #1 \right)} 
\newcommand{\bigb}[1]{\left[ #1 \right]} 
\newcommand{\bigc}[1]{\left\{ #1 \right\}} 

\newcommand{\floor}[1]{\left\lfloor #1 \right\rfloor}

\newcommand{\NN}{\mathbb N}
\newcommand{\ZZ}{\mathbb Z}
\newcommand{\cS}{\mathcal S}

\newtheorem{thm}{Theorem}[section]
\newtheorem{theorem}[thm]{Theorem}
\newtheorem{corollary}[thm]{Corollary}
\newtheorem{lemma}[thm]{Lemma}
\newtheorem{proposition}[thm]{Proposition}
\newtheorem{definition}[thm]{Definition}
\newtheorem{problem}[thm]{Problem}
\newtheorem{conjecture}[thm]{Conjecture}
\newtheorem*{claim}{Claim}
\newtheorem{fact}[thm]{Fact}
\newtheorem{remark}[thm]{Remark}
\newtheorem{example}[thm]{Example}

\title{Zero-sum subsequences in bounded-sum $\{-r,s\}$-sequences}
\author{Alec Sun}
\address[Alec Sun]{Department of Mathematics, Harvard University}
\email{sundogx@gmail.com}

\begin{document}

\begin{abstract}
    We study the problem of finding zero-sum blocks in bounded-sum sequences, which was introduced by Caro, Hansberg, and Montejano. Caro et al.\! determine the minimum $\{-1,1\}$-sequence length for when there exist $k$ consecutive terms that sum to zero. We determine the corresponding minimum sequence length when the set $\{-1,1\}$ is replaced by $\{-r,s\}$ for arbitrary positive integers $r$ and $s.$ This confirms a conjecture of theirs. We also construct $\{-1,1\}$-sequences of length quadratic in $k$ that avoid $k$ terms indexed by an arithmetic progression that sum to zero. This solves a second conjecture of theirs in the case of $\{-1,1\}$-sequences on zero-sum arithmetic subsequences. Finally, we give for sufficiently large $k$ a superlinear lower bound on the minimum sequence length to find a zero-sum arithmetic progression for general $\{-r,s\}$-sequences. 
\end{abstract}

\maketitle

\section{Introduction}\label{introduction}

The topics of this paper concern results antipodal to those on sequences in discrepancy theory \cite{C}, the study of deviations from uniformity in combinatorial settings. There is a famous theorem in discrepancy theory due to Roth \cite{R} on subsequences indexed by arithmetic progressions\footnote{Throughout this paper, an arithmetic progression in $[n]$ with $k$ terms and common difference $d$ will denote a set of numbers $\{a,a+d,\ldots,a+(k-1)d\}$ each of which is in $[n],$ and an arithmetic progression $A\subeq [n]$ can have any number of terms $k$ and any common difference $d.$} in $\{-1,1\}$-sequences.

\begin{theorem}[Roth]\label{roth}
Let $[n]$ denote the set $\{1,2,\ldots,n\}.$ For any positive integer $n$ and any function $f:[n]\to \{-1,1\}$ there exists an arithmetic progression $A\subeq [n]$ such that $$\abs{\sum_{x\in A} f(x)} \ge cn^{1/4}$$ for some positive constant $c.$
\end{theorem}

Matousek and Spencer \cite{MS} showed that the bound in Theorem \ref{roth} is sharp up to a constant factor. Another result in discrepancy theory regarding arithmetic subsequences\footnote{Here an \emph{arithmetic subsequence} refers to $k$ terms of a sequence corresponding to indices of a $k$-term arithmetic progression.} of $\{-1,1\}$-sequences is the proof of the Erd\H{o}s Discrepancy Conjecture, which states that for any sequence $f:\NN\to \{-1,1\}$ the discrepancy $$\sup_{n,d\in \NN} \abs{\sum_{j=1}^n f(jd)}$$ is unbounded, by Tao \cite{T}.

Caro, Hansberg, and Montejano \cite{CHM} consider the somewhat opposite direction. Instead of trying to maximize the quantity $$\abs{\sum_{x\in A} f(x)}$$ for subsequences $A\subeq [n]$ as one does in discrepancy theory, they attempt to minimize it. They introduce the following definitions:

\begin{definition}
Let $X$ denote any set. Given an integer function $f:X\to \ZZ$ and any subset $Y\subeq X,$ define $$f(Y) = \sum_{y\in Y} f(y).$$ We sometimes refer to $f(Y)$ as the \emph{weight} of $Y$ with respect to $f.$
\end{definition}

\begin{definition}
We say that $Y$ is a \emph{zero-sum set} with respect to $f$ if $Y$ has weight 0, namely $f(Y) = 0.$
\end{definition}

\begin{definition}
A \emph{$k$-block} is a set of $k$ consecutive integers on a given sequence of integers. A zero-sum $k$-block is a $k$-block that is also a zero-sum set.
\end{definition}

In general, a \emph{zero-sum problem} studies conditions needed to ensure that a given sequence has a zero-sum subsequence. One of the first theorems in this subject is the Erd\H{o}s-Ginzburg-Ziv Theorem \cite{EGZ}, which says that given a sequence of $2n-1$ elements of $\ZZ_n,$ there exists a subsequence of length $n$ that has zero weight. Furthermore, the number $2n-1$ is the smallest integer with this property.

The class of zero-sum problems has been extensively studied for abelian groups. Overviews of zero-sum problems have been written by Caro \cite{C2} as well as Gao and Geroldinger \cite{GG}. In the context of zero-sum subsequences over the integers, the results in this paper are related to results in, for example, \cite{AMSSV}, \cite{B}, \cite{CY}, and \cite{SST}.

A principal result of Caro et al.\! in \cite{CHM} deals with small-sum $k$-blocks in bounded-sum $\{-1,1\}$-sequences.

\begin{theorem}[\cite{CHM}, Theorem 2.3]\label{+-1-theorem}
Let $t$, $k$, and $q$ be integers such that $q\ge 0$, $0\le t<k$, and $t\equiv k\bmod 2.$ Let $s\in [0,t+1]$ be the unique integer satisfying $s\equiv q+\frac{k-t-2}{2} \bmod{t+2}.$ Then for any integer $n$ such that $$n\ge \max\bigc{ k, \frac{1}{2(t+2)}k^2 + \frac{q-s}{t+2}k - \frac{t}{2} + s}$$ and any function $f:[n]\to \{-1,1\}$ with $$\abs{\sum_{i=1}^n f(i)}\le q,$$ there is a $k$-block $B\subeq [n]$ with $$\abs{\sum_{y\in B}f(y)} \le t.$$
\end{theorem}

Setting $t=0$ in Theorem \ref{+-1-theorem} yields the following corollary regarding zero-sum $k$-blocks in bounded-sum $\{-1,1\}$-sequences:

\begin{corollary}[\cite{CHM}, Corollary 2.4]\label{+-1-corollary}
Let $k\ge 2$ be even and let $q\ge 0$ be an integer. Take $s\in \{0,1\}$ as the unique integer satisfying $s\equiv q + \frac{k-2}{2}\bmod 2.$ Then for any integer $n$ such that $$n\ge \max\bigc{k,\frac{k^2}{4} + \frac{q-s}{2}k+s}$$ and any function $f:[n]\to \{-1,1\}$ with $\abs{f([n])}\le q,$ there is a zero-sum $k$-block in $[n].$
\end{corollary}





Corollary \ref{+-1-corollary} can be extended to cover the range $q=o(n).$ The following is an infinite version of this phenomenon:

\begin{theorem}[\cite{CHM}, Theorem 2.10]\label{infinite-theorem}
Let $f:\ZZ^+\to \{-1,1\}$ be a function such that $\abs{f([n])} = o(n)$ when $n\to \infty.$ Then for every even $k\ge 2$ there are infinitely many zero-sum $k$-blocks.
\end{theorem}

Caro et al.\! note that there are applications of Theorem \ref{infinite-theorem} to two well-known number theoretic functions, Liouville's function and the Legendre symbol relating to quadratic residues and non-residues. The first application of Theorem \ref{infinite-theorem} in \cite{CHM} is a result on the Liouville function relating to the work of Hildebrand \cite{H} on sign patterns of this function in short intervals. The second application \cite{CHM} is to zero-sum blocks of consecutive primes when subjected to the Legendre symbol.

There are many possible natural generalizations of Corollary \ref{+-1-corollary} about zero sum $k$-blocks when the range of $f$ is replaced by another range, such as intervals $[x,y] = \{x,x+1,\ldots,y\}.$ Caro et al.\! mention that in most cases the existence of precisely $k$-consecutive zero-sum terms is not guaranteed. However, in the case where $f:[n]\to \{-r,s\}$ for arbitrary positive integers $r$ and $s,$ a result similar to Corollary \ref{+-1-corollary} can be deduced. Caro et al.\! have the following conjecture:

\begin{conjecture}[\cite{CHM}, Conjecture 5.4]\label{rs-conjecture}
Let $r$, $s$, and $k$ be positive integers such that $r+s$ divides $k.$ Then there exists a constant $c(r,s)$ such that if $$n\ge \frac{rs}{(r+s)^2}k^2 + c(r,s)k$$ then any function $f:[n]\to \{-r,s\}$ with $f([n]) = 0$ contains a zero-sum $k$-block.
\end{conjecture}

There is also the stronger question of finding the minimum such $n$ needed to guarantee the existence of a zero-sum $k$-block.

\begin{problem}[\cite{CHM}, Problem 1]\label{rs-problem}
Let $r$, $s$, and $k$ be positive integers such that $r+s$ divides $k.$ Determine the minimum value of $N(r,s,k)$ such that for $n\ge N(r,s,k)$ Conjecture \ref{rs-conjecture} holds.
\end{problem}


In Section \ref{zero-sum-blocks} we answer this stronger question with Theorem \ref{tight-rs-corollary}, thereby resolving Conjecture \ref{rs-conjecture} in the affirmative. Since the closed-form expression for $N(r,s,k)$ that we will prove is rather unwieldy, we define a notational shorthand which will be used throughout Section \ref{zero-sum-blocks}.

\begin{definition}
Define the function $$N_k(x,y,z) = \begin{cases}\bigp{\frac{xyk}{(x+y)^2} - \frac{x+yz}{x+y}}k + \frac{yk}{x+y}+z & z\le x \\ \bigp{\frac{xyk}{(x+y)^2} - \frac{x+x(x+y-z)}{x+y}}k + \frac{yk}{x+y}-(x+y-z) & z> x.\end{cases}$$
\end{definition}

\begin{theorem}\label{tight-rs-corollary}
Let $r$, $s$, and $k$ be positive integers such that $r<s$, $\gcd(r,s)=1$, and $r+s$ divides $k.$ Let $t\in [0, r+s-1]$ be the unique integer such that $\frac{sk}{r+s}-1+t\equiv 0\bmod r+s$ and let $t'\in [0,r+s-1]$ be the unique integer such that $\frac{rk}{r+s}-1+t'\equiv 0\bmod r+s.$ Then $$N(r,s,k) = \max\bigc{k, N_k(r,s,t), N_k(s,r,t')}.$$
\end{theorem}



Another direction of generalization is to replace the structure of blocks with that of arithmetic progressions. Studying the arithmetic progression case is motivated by the theorem of Roth, because Theorem \ref{roth} also deals with arithmetic progressions in $[n].$ As sequences of $k$ consecutive integers are $k$-term arithmetic progressions with common difference 1, it is natural to ask whether Corollary \ref{+-1-corollary} offers the best possible value for arithmetic progressions as well. Caro et al.\! remark that this is not the case. They believe the problem of finding precisely the corresponding minimum positive integer for arithmetic progressions is difficult.

\begin{problem}[\cite{CHM}, Problem 2]
Let $r$, $s$, and $k$ be positive integers such that $r+s$ divides $k.$ Determine the minimum value $M(r,s,k)$ such that if $n\ge M(r,s,k)$ then for every function $f:[n]\to \{-r,s\}$ with $f([n])=0$ there exists a $k$-term arithmetic progression $A\subeq [n]$ with $f(A) = 0.$
\end{problem}

In the context of general subsequences such as in \cite{AMSSV}, \cite{B}, \cite{CY}, and \cite{SST}, as opposed to subsequences with indices that are consecutive or in arithmetic progression studied in this paper, it seems that the upper bounds on lengths of sequences avoiding $k$-length zero-sum subsequences are linear in $k.$ For example, see the bounds proven by Augspurger, Minter, Shoukry, Sissokho, and Voss in \cite{AMSSV} or by Berger in \cite{B}.

On the other hand, when one restricts the subsequences to $k$-blocks or $k$-term arithmetic progressions, it appears that the bounds on the length become quadratic in $k$, such as in Theorem \ref{tight-rs-corollary}, or at least superlinear in $k.$ Along these lines, Caro et al.\! conjecture that the asymptotic bound on $n$ for a zero-sum subsequence indexed by a $k$-term arithmetic progression should remain quadratic in $k.$

\begin{conjecture}[\cite{CHM}, Conjecture 5.6]\label{quadratic-conjecture}
Let $r$, $s$, and $k$ be positive integers such that $r+s$ divides $k.$ There are positive constants $c_1 = c_1(r,s)$ and $c_2=c_2(r,s)$ such that $$c_1 k^2 \le M(r,s,k)\le c_2 k^2.$$
\end{conjecture}

Clearly a bound obtained for zero-sum $k$-blocks in the resolution of Conjecture \ref{rs-conjecture} serves as an upper bound for $M(r,s,k).$ In particular, for any $c(r,s)$ for which Conjecture \ref{rs-conjecture} is true we have $$M(r,s,k)\le \frac{rs}{(r+s)^2} k^2 + c(r,s)k.$$ 
Hence the open problem is the lower bound. In Section \ref{arithmetic} we resolve Conjecture \ref{quadratic-conjecture} in the case $r=s=1,$ namely for $\{-1,1\}$-sequences.

\begin{theorem}\label{quadratic-theorem}
We have $$\frac{1}{6}k^2 + O(k) \le M(1,1,k) \le \frac{1}{4}k^2 + O(k)$$ for all even $k.$\footnote{Recall that in order for a zero-sum subsequence of $k$ elements to exist, $k$ must be even. }
\end{theorem}

Our results in Section \ref{superlinear} are a superlinear lower bound for sufficiently large $k$ on $M(r,s,k)$ for $\{-r,s\}$-sequences and a construction that shows tightness for the constant $c_2(1,1)$ defined in Conjecture \ref{quadratic-conjecture}.

\begin{theorem}\label{not-quadratic-theorem}
Let $r$, $s$, and $k$ be positive integers such that $r+s$ divides $k.$ Then $$\Omega_{r,s}\bigp{k^{1.475}}\le M(r,s,k)\le \frac{rs}{(r+s)^2}k^2 + O(k)$$ for sufficiently large $k.$
\end{theorem}

Finally, in Section \ref{open} we mention some open problems with regards to improving the bounds on the length of $\{-r,s\}$-sequences that avoid zero-sum arithmetic subsequences.



\subsection*{Acknowledgements}
This research was funded by NSF/DMS grant 1659047 and NSA grant H98230-18-1-0010 as part of the 2019 Duluth Research Experience for Undergraduates (REU) program. The author would like to thank Joseph Gallian for running the program, as well as Aaron Berger and Brian Lawrence for comments on the paper. Finally, the author would like to thank the anonymous reviewers for helpful comments, corrections, and suggestions.

\section{Existence of zero-sum blocks in $\{-r,s\}$-sequences}\label{zero-sum-blocks}

Recall the statement of Theorem \ref{+-1-theorem}, the principal result of Caro et al.\! in \cite{CHM}, which gives a lower bound on $n$ for when a small-sum $k$-block in a bounded-sum $\{-1,1\}$-sequence always exists. It turns out that Theorem \ref{+-1-theorem} is the best possible bound for the parameters involved in the sense that for $$n = \frac{1}{2(t+2)}k^2 + \frac{q-s}{t+2}k-\frac{t}{2}+s-1$$ there are examples of functions having the highest possible value for $f([n]),$ namely $q,$ such that no $k$-block $B\subeq [n]$ satisfies $\abs{f(B)}\le t.$ In particular there is the following theorem:

\begin{theorem}[\cite{CHM}, Theorem 2.9]\label{+-1-extremal-theorem}
Let $k$, $t$, and $q$ be integers such that $q\ge 0$, $0\le t<k$, $t\equiv k\bmod 2$, and $$k < \frac{1}{2(t+2)}k^2 + \frac{q-s}{t+2}k - \frac{t}{2} + s,$$ where $s\in [0,t+1]$ is the unique integer satisfying $s\equiv q + \frac{k-t-2}{2}\bmod t+2.$ Then, for $$n= \frac{1}{2(t+2)}k^2 + \frac{q-s}{t+2}k - \frac{t}{2} + s-1,$$ there exists a function $f:[n]\to \{-1,1\}$ satisfying $\abs{f([n])}=q$ and $\abs{f(B)}>t$ for all $k$-blocks $B\subeq [n].$
\end{theorem}

Yet it is not true that all values of $n$ greater than $k$ and less than $$\frac{1}{2(t+2)}k^2 + \frac{q-s}{t+2}k-\frac{t}{2}+s$$ exhibit constructions that avoid zero-sum $k$-blocks. For example, taking $t=0$ and $q=0$ in the statement of Theorem \ref{+-1-extremal-theorem}, we prove the following proposition that guarantees a zero-sum $k$-block in a zero-sum sequence of length $2k$:

\begin{proposition}
Let $k\ge 2$ be even. Then for any function $f:[n]\to \{-1,1\}$ with $n=2k$ and $f([n]) = 0,$ there is a zero-sum $k$-block in $[n].$
\end{proposition}

\begin{proof}
Split the set $[n]$ into two disjoint $k$-blocks $B_1$ and $B_2.$ If $f(B_1) = 0$ then we are done. Otherwise assume that $f(B_1)\neq 0.$ We know that
\begin{align*}
    f(B_2) &= f([n]) - f(B_1)
    = 0 - f(B_1)
    = -f(B_1),
\end{align*}
which implies that $f(B_1)$ and $f(B_2)$ have opposite sign. Now we apply the ``Interpolation Lemma'', which is Lemma \ref{interpolation} below, with $t=0.$ This lemma implies the existence of a zero-sum $k$-block in $[n].$
\end{proof}

There are two goals of this section. The first goal is to resolve Conjecture \ref{rs-conjecture} posed by Caro et al.\! regarding the existence of zero-sum blocks in $\{-r,s\}$-sequences with a constant $c(r,s)\approx \frac{\abs{r-s}}{r+s}.$ The proof is analogous to that of Theorem \ref{+-1-theorem}. The second goal is to solve Problem \ref{rs-problem} by proving Theorem \ref{tight-rs-corollary}.

The condition that $r+s$ divides $k$ in Conjecture \ref{rs-conjecture} is a technical condition to ensure that there exist $k$ numbers each of which is either $-r$ or $s$ that sum to 0. In particular, if $\gcd(r,s) = 1,$ then one can see that the condition that $r+s$ divides $k$ is necessary for zero-sum $k$-blocks to exist. For the rest of this paper we assume that $\gcd(r,s)=1$ when dealing with $\{-r,s\}$-sequences. We do so because we can divide all integers by $\gcd(r,s)$ to produce two relatively prime integers noting that the relevant problem is equivalent under scalar multiplication of all terms. In this section we will make use of the following fact:

\begin{fact}\label{rsk-fact}
If $a_1,a_2,\ldots,a_k\in \{-r,s\},$ then $a_1+a_2+\cdots + a_k \equiv sk\bmod r+s.$ In particular, if $k\equiv 0\bmod r+s$ then $a_1+a_2+\cdots + a_k \equiv 0\bmod r+s.$
\end{fact}

In our argument we will adapt the following ``Interpolation Lemma'', which is an important ingredient in Caro et al.'s proof of Theorem \ref{+-1-theorem}:

\begin{lemma}[\cite{CHM}, Lemma 2.1]\label{interpolation}
Let $t$, $k$, and $n$ be integers such that $t\equiv k\bmod 2$ and $\abs{t}<k\le n.$ Let $f:[n]\to \{-1,1\}$ be any function. If there are $k$-blocks $S$ and $T$ in $[n]$ such that $f(S)<t$ and $f(T)>t,$ then there is a $k$-block $B$ in $[n]$ with $f(B)=t.$
\end{lemma}

We begin the proof of Conjecture \ref{rs-conjecture} by proving an analog of Lemma \ref{interpolation}:

\begin{lemma}[Interpolation Lemma]\label{interpolation-lemma}
Let $r$, $s$, $k$, and $n$ be positive integers such that $k\le n$ and $r+s$ divides $k,$ and let $f:[n]\to \{-r,s\}.$ If there are $k$-blocks $S$ and $T$ in $[n]$ such that $f(S)<0$ and $f(T)>0,$ then there is a $k$-block $B$ in $[n]$ with $f(B)=0.$
\end{lemma}
\begin{proof}
Denote the $n-k+1$ $k$-blocks in $[n]$ by $B_1,B_2,\ldots,B_{n-k+1},$ where $B_i = [i,i+k-1]$ for $1\le i\le n-k+1.$ Let $S = B_s$ and $T = B_t$ for some indices $s$ and $t.$ By Fact \ref{rsk-fact}, we know that the sum of every $k$-block is $0\bmod r+s,$ so $f(B_i) \equiv 0\bmod r+s$ for all $i\in [1,n-k+1].$ Also note that $\abs{f(B_i) - f(B_{i+1})} \le r+s$ since $B_i$ and $B_{i+1}$ differ in exactly two elements. We conclude by an Intermediate Value Theorem argument that there exists a zero-sum $k$-block in $[n].$
\end{proof}

We now prove a slight generalization of Conjecture \ref{rs-conjecture} in which the weight of the entire sequence, namely $f([n]),$ needs only be bounded in absolute value by a constant $q,$ as it is done similarly in Corollary \ref{+-1-corollary} and Theorem \ref{+-1-extremal-theorem}. Before we state this generalization, we need the following simple proposition, which will be used throughout Section \ref{zero-sum-blocks}:

\begin{proposition}\label{rs-claim1}
For positive integers $r$ and $s,$ let $f:[n]\to \{-r,s\}$ be a function and let $B\subeq [n]$ be any block. If $f(B)\ge 0$ then $\abs{f(B)}\le s\cdot \abs{B}$ and if $f(B)\le 0$ then $\abs{f(B)}\le r\cdot \abs{B}.$
\end{proposition}

\begin{proof}
Note that we have $-r \leq f(i) \leq s$ for all $i$, which implies that $-r \cdot \abs{B} \leq f(B) \leq s \cdot \abs{B}$ by summing over all $\abs{B}$ terms in $B.$ The right inequality implies that $\abs{f(B)}\le s\cdot \abs{B}$ when $f(B)\ge 0,$ and the left inequality when multiplied by $-1$ implies that $\abs{f(B)}\le r\cdot \abs{B}$ when $f(B)\le 0.$
\end{proof}

\begin{theorem}\label{rs-theorem}
Let $r$, $s$, and $k$ be positive integers such that $r+s$ divides $k$ and let $q\ge 0$ be an integer. Then if $$n\ge \max\bigc{k, k\floor{\frac{q-r}{r+s}+\frac{rsk}{(r+s)^2}}+\frac{sk}{r+s}+\frac{r}{s},k\floor{\frac{q-s}{r+s}+\frac{rsk}{(r+s)^2}}+\frac{rk}{r+s}+\frac{s}{r}},$$ every function $f:[n]\to \{-r,s\}$ with $\abs{f([n])} \le q$ contains a zero-sum $k$-block.
\end{theorem}

\begin{proof}
If there are $k$-blocks $S$ and $T$ in $[n]$ such that $f(S) < 0$ and $f(T)>0$ then we are done by Lemma \ref{interpolation-lemma}. If $f(S)=0$ for some $k$-block $S$ we are done as well. Hence we can assume for the sake of contradiction that $f(S)>0$ for all $k$-blocks $S$ or $f(S)<0$ for all $k$-blocks $S.$ By Fact \ref{rsk-fact} we in fact have either $f(S)\ge r+s$ or $f(S)\le -r-s$ in the respective cases. We prove the first case $f(S)\ge r+s$ and reduce the case $f(S)\le -r-s$ to the first.
    
\begin{figure}
\centering
\includegraphics[width = 0.5\textwidth]{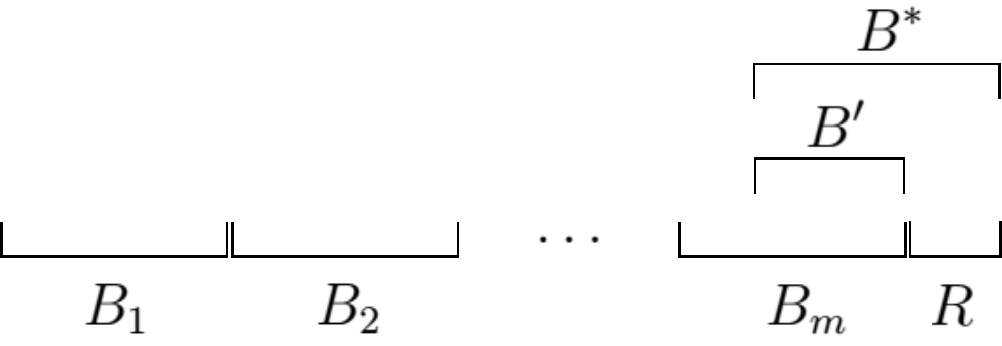}
\caption{Diagram depicting the $k$-blocks $B_1,B_2,\ldots,B_m,B^*$ and the blocks $B',R.$}
\label{split}
\end{figure}
    
\begin{enumerate}
    \item[\textbf{Case 1.}] \textit{We have $f(S)\ge r+s$ for all $k$-blocks $S.$}
    
    Write $n=mk+a$ where the quotient $m$ and remainder $a$ are positive integers and $0\le a\le k-1.$ Split $[n]$ into $m$ disjoint consecutive $k$-blocks $B_1,B_2,\ldots,B_m$ where $B_i = [k(i-1)+1,k(i-1)+k]$ and a remainder block $R = [n-a+1,n]$ that is potentially empty. See Figure \ref{split}.
    
    We know by hypothesis that $n\ge k,$ meaning $m\ge 1.$ We also know that $f(B_i)\ge r+s$ for all $i,$ implying
    \begin{align*}
        q &\ge f([n])
        \ge (r+s)m + f(R)
    \end{align*}
    and thus
    \begin{equation}\label{Rqrsm}
        f(R) \le q - (r+s)m.
    \end{equation}
    Let $B^*$ denote the rightmost $k$-block $[n-k+1,n]$ and let $B' = B^*\cap B_m$ so that $B^* = B' \cup R.$ We know that $$f(B') + f(R) = f(B^*) \ge r+s.$$ Using \eqref{Rqrsm}, this implies
    \begin{align*}
        f(B') &\ge r+s-f(R)
        \ge (r+s)(m+1)-q.
    \end{align*}
    We can see that $$n\ge k\floor{\frac{q-r}{r+s}+\frac{rsk}{(r+s)^2}}+\frac{sk}{r+s}+\frac{r}{s}$$ by hypothesis. This bound on $n$ is a sufficient condition for \textbf{Case 1} to work, and we will show later in our proof of \textbf{Case 2} that this condition with the variables $r$ and $s$ switched holds in order to reduce \textbf{Case 2} to \textbf{Case 1}. For $$n\ge k\floor{\frac{q-r}{r+s}+\frac{rsk}{(r+s)^2}}+\frac{sk}{r+s}+\frac{r}{s}$$ we have
    \begin{equation}\label{rs-fact2}
    m\ge \floor{\frac{q-r}{r+s}+\frac{rsk}{(r+s)^2}}.
    \end{equation}
    We split into two cases.
    \begin{enumerate}
        \item[\textbf{Case 1.1.}] \textit{We have $$m = \floor{\frac{q-r}{r+s}+\frac{rsk}{(r+s)^2}}.$$}
        
        In this case we see that $a\ge \frac{sk}{r+s}+\frac{r}{s},$ otherwise $m$ as defined above would be greater. We note that
        \begin{align*}
            f(B') &\ge (r+s)(m+1)-q
            \\&> (r+s)\bigp{\frac{q-r}{r+s}+\frac{rsk}{(r+s)^2}} - q
            \\&= \frac{rsk}{r+s}-r
            \\&\ge 0.
        \end{align*}
        Applying Proposition \ref{rs-claim1} yields that
        \begin{align*}
            \abs{B'} &\ge \frac{\abs{f(B')}}{s}
            > \frac{\frac{rsk}{r+s}-r}{s}
        \end{align*}
        and
        \begin{align*}
            a &= \abs{R}
            = k - \abs{B'}
            < \frac{sk}{r+s}+\frac{r}{s},
        \end{align*}
        which contradicts $a\ge \frac{sk}{r+s}+\frac{r}{s}$ above. We conclude the existence of a zero-sum $k$-block.
        
        \item[\textbf{Case 1.2.}] \textit{We have $$m\ge \floor{\frac{q-r}{r+s}+\frac{rsk}{(r+s)^2}}+1.$$}
        
        We can compute that
        \begin{align*}
            f(R) &\le q-(r+s)m
            \\&< q- (r+s)\bigp{\frac{q-r}{r+s}+\frac{rsk}{(r+s)^2}}
            \\&= r - \frac{rsk}{r+s}
            \\&\le 0,
        \end{align*}
        and that
        \begin{align*}
            f(B') &\ge (r+s)(m+1)-q
            > (r+s)m-q
            > 0.
        \end{align*}
        Now apply Proposition \ref{rs-claim1} to get
        \begin{align*}
            \abs{B'} &\ge \frac{\abs{f(B')}}{s}
            \ge \frac{(r+s)(m+1)-q}{s} \\
            \abs{R} &\ge \frac{\abs{f(R)}}{r}
            \ge \frac{(r+s)m-q}{r} \\
            \abs{B^*} &= \abs{B'} + \abs{R}
            \\&\ge \frac{(r+s)(m+1)-q}{s} + \frac{(r+s)m-q}{r}
            \\&= \frac{(r+s)^2}{rs} m + \frac{(r+s)(r-q)}{rs}
            \\&> \frac{(q-r)(r+s)}{rs}+ k+ \frac{(r+s)(r-q)}{rs}
            \\&= k,
        \end{align*}
        which contradicts the fact that $\abs{B^*} = k.$ We conclude the existence of a zero-sum $k$-block.
    \end{enumerate}
    
    \item[\textbf{Case 2.}] \textit{We have $f(S)\le -r-s$ for all $k$-blocks $S.$}
    
    We reduce this case to \textbf{Case 1}. Negate every term of the sequence $f(1),f(2),\ldots,f(n)$ by defining the function $g:[n]\to \{-s,r\}$ such that $g(i) = -f(i).$ Note that $g(\abs{n}) = -f(\abs{n}),$ implying that $\abs{g([n])}\le q.$ Note also that all $k$-blocks $S\subeq [n]$ satisfy $g(S)\ge r+s.$ Finally, we see that
    \begin{align*}
        n &\ge \max\bigc{k\floor{\frac{q-r}{r+s}+\frac{rsk}{(r+s)^2}}+\frac{sk}{r+s}+\frac{r}{s},k\floor{\frac{q-s}{r+s}+\frac{rsk}{(r+s)^2}}+\frac{rk}{r+s}+\frac{s}{r}}
        \\&\ge k\floor{\frac{q-s}{r+s}+\frac{rsk}{(r+s)^2}}+\frac{rk}{r+s}+\frac{s}{r},
    \end{align*}
    meaning that we have reduced this case to \textbf{Case 1}.

\end{enumerate}

Having analyzed both cases, we conclude the existence of a zero-sum $k$-block. This completes the proof of Theorem \ref{rs-theorem}.
\end{proof}

One can now see that Theorem \ref{rs-theorem} implies Conjecture \ref{rs-conjecture} with constant $c(r,s)\approx \frac{\abs{r-s}}{r+s}$ by setting $q=0.$

A natural question to ask is how tight our bound is for zero-sum $k$-blocks in zero-sum sequences. Caro, Hansberg, and Montejano claim that the general case of $\{-r,s\},$ as opposed to the special case of $\{-1,1\},$ appears to be significantly more complicated not only in terms of characterizing the extremal sequences but also in finding the exact value of $N(r,s,k)$ as defined in Problem \ref{rs-problem}. For example, to exhibit constructions of zero-sum $\{-r,s\}$-sequences that avoid zero-sum $k$-blocks, one has to be mindful of the fact that a sequence of $n$ numbers from $\{-r,s\}$ where $\gcd(r,s)=1$ can only sum to 0 if $r+s$ divides $n.$ In light of this fact, Problem \ref{rs-problem} essentially asks for the minimum value of $N(r,s,k)$ up to a multiple of $r+s.$

We are now ready to prove Theorem \ref{tight-rs-corollary}, which gives the exact value of $N(r,s,k).$ We will first show the lower bound for $N(r,s,k)$ using a construction in Lemma \ref{rs-construction} and then show that this bound is tight. The setup will be relatively prime positive integers $r$ and $s$ and a function $f:[n]\to \{-r,s\}.$ Let us assume that $r\neq s$ because otherwise $r=s=1,$ and this case is already established in \cite{CHM}.

\begin{lemma}\label{rs-construction}
Let $r$, $s$, and $k$ be positive integers such that $\gcd(r,s)=1$ and $r+s$ divides $k.$ Let $t\in [0, r+s-1]$ be the unique integer such that $\frac{sk}{r+s}-1+t\equiv 0\bmod r+s.$ Then there exists a function $f:[n]\to \{-r,s\}$ with $n = N_k(r,s,t)-1$ such that $f([n]) = 0$ and $f(B)\neq 0$ for any $k$-block $B\subeq [n].$
\end{lemma}

\begin{proof}
We first describe the constructions of $f$ in each of the two cases $t\le r$ and $t>r$ and prove that in each case we have $f([n])=0.$
\begin{enumerate}

    \item[\textbf{Case 1.}] \textit{We have $t\le r$.}
    
    Split the range $[n]$ into $$b = \frac{rsk}{(r+s)^2} - \frac{r+st}{r+s}$$ disjoint consecutive $k$-blocks and a remainder block of length $$\frac{sk}{r+s}-1+t < \frac{sk}{r+s} + \frac{rk}{r+s} = k.$$ For each of the $b$ $k$-blocks, set the first $\frac{sk}{r+s}-1$ terms to $-r$ and the last $\frac{rk}{r+s}+1$ terms to $s.$ We can verify that the sum of the terms in each block is $$\bigp{\frac{k}{r+s}s-1}(-r) + \bigp{\frac{k}{r+s}r+1}s = r+s.$$ Setting the first $\frac{sk}{r+s}-1$ terms of the remainder block to $-r$ and the last $t$ terms to $s,$ we compute
    \begin{align*}
        f([n]) &= (r+s)b - r\bigp{\frac{sk}{r+s}-1} + st
        \\&= (r+s)\bigp{\frac{rsk}{(r+s)^2} - \frac{r+st}{r+s}} - r\bigp{\frac{sk}{r+s}-1} + st
        \\&= 0.
    \end{align*}
    
    \item[\textbf{Case 2.}] \textit{We have $t>r$.}
    
    Proceed with the same construction as in \textbf{Case 1} with $$b = \frac{rsk}{(r+s)^2} - \frac{r+r(r+s-t)}{r+s}$$ until the step that involves setting the terms of the remainder block. Instead, set all $$\frac{sk}{r+s}-1-(r+s-t) < k$$ terms of the remainder block to $-r.$ We compute
    \begin{align*}
        f([n]) &= (r+s)b - r\bigp{\frac{sk}{r+s}-1-(r+s-t)}
        \\&= (r+s)\bigp{\frac{rsk}{(r+s)^2} - \frac{r+r(r+s-t)}{r+s}} - r\bigp{\frac{sk}{r+s}-1-(r+s-t)}
        \\&= 0.
    \end{align*}
\end{enumerate}
Having analyzed both cases, we conclude that $f([n]) = 0.$

Now we show that in both constructions there is no $k$-block $B\subeq [n]$ such that $f(B) = 0.$ In fact, we make the following claim:
    \begin{claim}\label{k-block-sum}
    Every $k$-block $B\subeq [n]$ has $f(B) \ge r+s.$
    \end{claim}
    To prove this claim, note that in both constructions in the proof of Lemma \ref{rs-construction}, the terms equal to $-r$ form disjoint blocks of length at most $\frac{sk}{r+s}-1.$ Furthermore, any two consecutive such blocks are separated by at least $\frac{rk}{r+s}+1$ terms equal to $s.$ Therefore, among any $$\frac{sk}{r+s}-1 + \frac{rk}{r+s}+1 = k$$ consecutive terms, there are at most $\frac{sk}{r+s}-1$ terms equal to $-r.$ We conclude that any $k$-block has weight at least $$-r\bigp{\frac{sk}{r+s}-1} + s\bigp{\frac{rk}{r+s}+1} = r+s.$$
    Thus this claim holds, finishing the proof of Lemma \ref{rs-construction}.
    \end{proof}

\begin{corollary}\label{rs-construction-corollary}
Let $r$, $s$, and $k$ be positive integers with $\gcd(r,s)=1.$ Let $t'\in [0, r+s-1]$ be the unique integer such that $\frac{rk}{r+s}-1+t'\equiv 0\bmod r+s.$ Then there exists a function $f:[n]\to \{-r,s\}$ with $n = N_k(s,r,t')-1$ such that $f([n]) = 0$ and $f(B)\neq 0$ for any $k$-block $B\subeq [n].$
\end{corollary}
\begin{proof}
By Lemma \ref{rs-construction} there exists a function $g:[n]\to \{-s,r\}$ where $n = N_k(s,r,t)-1$ such that $g([n]) = 0$ and $g(B)\neq 0$ for any $k$-block $B\subeq [n].$ Let $f:[n]\to \{-r,s\}$ be defined such that $f(i) = -g(i)$ for all $i\in [n].$ Then it is easy to see that $f([n]) = 0$ and $f(B)\neq 0$ for any $k$-block $B\subeq [n]$ as well.
\end{proof}

We now prove Theorem \ref{tight-rs-corollary} in Section \ref{introduction}, which shows that the constructions given in Lemma \ref{rs-construction} and Corollary \ref{rs-construction-corollary} are tight.

\begin{proof}[Proof of Theorem \ref{tight-rs-corollary}]
Note that $N(r,s,k)\ge k$ because a sequence of length less than $k$ does not even contain a $k$-block, let alone a zero-sum $k$-block. Lemma \ref{rs-construction} and Corollary \ref{rs-construction-corollary} show the lower bounds $N(r,s,k)\ge N_k(r,s,t)$ and $N(r,s,k)\ge N_k(s,r,t')$ respectively. Hence it suffices to prove that for all functions $f:[n]\to \{-r,s\}$ with $$n\ge \max\{k, N_k(r,s,t), N_k(s,r,t')\}$$ and $f([n]) = 0,$ there exists a $k$-block $B\subeq [n]$ for which $f(B) = 0.$

Fix such a $n$ satisfying $$n\ge \max\{k, N_k(r,s,t), N_k(s,r,t')\}$$ and assume for the sake of contradiction that all $k$-blocks $B\subeq [n]$ satisfy $f(B)\neq 0.$ By Lemma \ref{interpolation-lemma}, either $f(B)\ge r+s$ for all $B$ or $f(B)\le -r-s$ for all $B.$

\begin{enumerate}
    \item[\textbf{Case 1.}] \textit{We have $f(B)\ge r+s$ for all $k$-blocks $B.$}
    
    The conditions $f([n])=0$ and $\gcd(r,s)=1$ imply that $r+s$ divides $n,$ so in fact $$n\ge N_k(r,s,t)-1+r+s.$$ Split the range $[n]$ into $b$ disjoint consecutive $k$-blocks starting from the first term and a remainder block $R$ consisting of $a = \abs{R}$ terms, where $0\le a\le k-1.$ From the hypothesis on $n$ we see that $$b \ge \begin{cases}\frac{rsk}{(r+s)^2} - \frac{r+st}{r+s} & t\le r \\ \frac{rsk}{(r+s)^2} - \frac{r+r(r+s-t)}{r+s} & t> r.\end{cases}$$ Since each of the $b$ $k$-blocks has weight at least $r+s,$ we see that the weight of the first $bk$ terms is at least $b(r+s),$ implying that $f(R)\le -b(r+s).$ We also know that because $r+s$ divides $k$ that the weight of each of the $b$ $k$-blocks is divisible by $r+s,$ implying that $f(R)$ is also divisible by $r+s$ because $r+s$ divides $f([n]) = 0.$ Since $-r,s\equiv s\bmod r+s$ we see that $f(R)\equiv sa\bmod r+s.$ The fact that $\gcd(r,s) = 1$ means that $\gcd(s,r+s)=1$ as well, implying that $r+s$ divides $a.$
    
    If there are at least $\frac{sk}{r+s}$ terms equal to $-r$ in $R$, then there would be at least $\frac{sk}{r+s}$ in the rightmost $k$-block $B^* = [n-k+1,n]\subset [n],$ implying that
    \begin{align*}
        f(B^*) \le -r\bigp{\frac{sk}{r+s}} + s\bigp{k - \frac{sk}{r+s}} \le 0,
    \end{align*}
    which is a contradiction. Hence there are at most $\frac{sk}{r+s}-1$ terms equal to $-r$ in $R.$
    
    \begin{enumerate}
        \item[\textbf{Case 1.1.}] \textit{We have $t\le r.$}
        
        We split into two subcases.
    
    \begin{enumerate}
        \item[\textbf{Case 1.1.1.}] \textit{We have $$b = \frac{rsk}{(r+s)^2} - \frac{r+st}{r+s}.$$}
        
        In this subcase we must have $$a=\frac{sk}{r+s}-1+t +r+s< k.$$ There are at most $\frac{sk}{r+s}-1$ terms equal to $-r$ in $R,$ implying that
        \begin{align*}
            f(R) &\ge -r\bigp{\frac{sk}{r+s}-1} + s\bigp{t+r+s}
            \\&= -\frac{rsk}{r+s} + r+st+s(r+s)
            \\&> -\frac{rsk}{r+s} + r+st
            \\&= -b(r+s),
        \end{align*}
        contradicting $f(R)\le -b(r+s).$ We conclude the existence of a zero-sum $k$-block in $[n].$
        
        \item[\textbf{Case 1.1.2.}] \textit{We have $$b \ge \bigp{\frac{rsk}{(r+s)^2} - \frac{r+st}{r+s}} + 1.$$}
        
        In this subcase we in fact have $$b = \bigp{\frac{rsk}{(r+s)^2} - \frac{r+st}{r+s}} + 1$$ and $$a = \frac{sk}{r+s}-1+t+r+s-k.$$ This is because $$\frac{sk}{r+s}-1+t+r+s-k<k,$$ so if $b$ is any larger then $a$ would be negative. By Proposition \ref{rs-claim1}, we see that the number of terms in $R$ equal to $-r$ is at least $$\frac{b(r+s)}{r} \ge \frac{\frac{rsk}{r+s}-st+s}{r}= \frac{sk}{r+s}+\frac{s-st}{r}.$$
        Hence we must have
        \begin{align*}
            a &\ge \frac{sk}{r+s} +\frac{s-st}{r} \\
            -1 + t + r + s - k &\ge \frac{s-st}{r} \\
            t+r-1 &\ge \frac{rk}{r+s} \\
            k &< 2(r+s),
        \end{align*}
        implying that $k=r+s$ because $r+s$ divides $k.$ This means that $a=s-1+t.$ We know that there are at most $$\frac{sk}{r+s}-1=s-1$$ terms equal to $-r$ in $R$, meaning that
        \begin{align*}
            f(R) &\ge -r(s-1) + st
            \\&= -rs + r + st
            \\&> -rs - s + st
            \\&= -b(r+s),
        \end{align*}
        contradicting $f(R)\le -b(r+s).$ We conclude the existence of a zero-sum $k$-block in $[n].$
    
    \end{enumerate}
    \item[\textbf{Case 1.2.}] \textit{We have $t> r.$}
    
    In this case we have $$b \ge \frac{rsk}{(r+s)^2} - \frac{r+r(r+s-t)}{r+s}.$$ If this inequality is strict, then this implies that
    \begin{equation}\label{ask}
        a\le \frac{sk}{r+s} - 1 + t - k.
    \end{equation}
    By Proposition \ref{rs-claim1} we see that the number of terms in $R$ equal to $-r$ is at least
    \begin{align*}
        \frac{b(r+s)}{r} &\ge \frac{sk}{r+s} +\frac{s}{r}- (r+s-t).
    \end{align*}
    So we have the inequality
    \begin{align*}
        a &\ge \frac{sk}{r+s} +\frac{s}{r}- (r+s-t).
    \end{align*}
    Combining with \eqref{ask}, we obtain $$k \le r+s - \frac{s}{r} - 1 < r+s.$$
    which is a contradiction. Hence $$b = \frac{rsk}{(r+s)^2} - \frac{r+r(r+s-t)}{r+s}$$ and $$a = \frac{sk}{r+s} - 1 + t.$$ We compute
    \begin{align*}
        f(R) &\ge -r\bigp{\frac{sk}{r+s} - 1} + st
        \\&= -\frac{rsk}{r+s} + r + (r+s)t - rt
        \\&> -\frac{rsk}{r+s} + r + r(r+s-t)
        \\&= -b(r+s),
    \end{align*}
    contradicting $f(R) \le -b(r+s).$ We conclude the existence of a zero-sum $k$-block in $[n].$
    \end{enumerate}
    
    \item[\textbf{Case 2.}] \textit{We have $f(B)\le -r-s$ for all $k$-blocks $B.$}
    
    As in the proof of Theorem \ref{rs-theorem}, we will reduce this case to \textbf{Case 1}. Construct a function $g:[n]\to \{-s,r\}$ such that $g(i) = -f(i)$ for all $i\in [n].$ Note that any $k$-block $B\subeq [n]$ has $g(B)\ge r+s,$ so $g$ satisfies the conditions of \textbf{Case 1}. By \textbf{Case 1}, there is a zero-sum $k$-block in the sequence defined by $g$ for any $n\ge N_k(s,r,t'),$ where $t'\in [0, r+s-1]$ is the unique integer such that $\frac{rk}{r+s} - 1 + t' \equiv 0\bmod r+s.$ Hence there is a zero-sum $k$-block in the sequence defined by $f$ for $n\ge N_k(s,r,t').$
\end{enumerate}

Having reached a contradiction in both cases with the assumption that all $k$-blocks $B\subeq [n]$ satisfy $f(B)\neq 0,$ we conclude the existence of a zero-sum $k$-block for $$n\ge \max\{k, N_k(r,s,t), N_k(s,r,t')\}.$$ This completes the proof of Theorem \ref{tight-rs-corollary}.
\end{proof}

\begin{remark}
Theorem \ref{tight-rs-corollary} precisely determines the constant $N(r,s,k),$ thus resolving Problem \ref{rs-problem} in Section \ref{introduction}. Note that the upper bound $$N(r,s,k)\le \max\{k, N_k(r,s,t), N_k(s,r,t')\}$$ in Theorem \ref{tight-rs-corollary} does not follow from Theorem \ref{rs-theorem} despite some similarities in their proofs. Nevertheless, for some specific values of $r$, $s$, and $k,$ one can verify that the two lower bounds for $n$ to guarantee a zero-sum $k$-block are the same in both theorems.
\end{remark}

Finally, we extend Theorem \ref{rs-theorem} to cover infinite sequences, generalizing Theorem \ref{infinite-theorem} which deals with the $\{-1,1\}$ case. The proof is entirely analogous.

\begin{theorem}\label{infinite-rs-theorem}
Let $f:\ZZ^+\to \{-r,s\}$ be a function such that $\abs{f([n])} = o(n)$ when $n\to \infty.$ Then for every $k$ such that $r+s$ divides $k,$ there are infinitely many zero-sum $k$-blocks.
\end{theorem}

\begin{proof}
Fix $n_0>1$ satisfying $n_0\equiv 1\bmod k$ and consider $$n>(k\cdot \max\{r,s\} + 2)(n_0-1)\ge n_0$$ such that $n\equiv 0\bmod k$ and $\abs{f([n])}\le \frac{n}{k}.$ Note that such a $n$ exists because $\abs{f([n])} = o(n)$ when $n\to \infty.$ Consider the partition of $[n_0, n]$ into $h = \frac{n-n_0+1}{k}$ disjoint $k$-blocks $B_1,B_2,\ldots,B_h.$ Suppose that $f(B_i)\neq 0$ for all $1\le i\le h.$ We know by Lemma \ref{interpolation-lemma} that if $f(B_i)<0<f(B_j)$ for some $1\le i,j\le h$ then there will be a zero-sum $k$-block. Hence we can assume without loss of generality that all blocks $B_i$ satisfy $f(B_i)\ge 1,$ and hence $f(B_i)\ge r+s\ge 2.$ Then we have
\begin{align*}
    2h
    &\le \sum_{i=1}^h f(B_i)
    \\&\le \abs{\sum_{j=n_0}^n f(j)}
    \\&\le \abs{f([n])} + \max\{r,s\}\cdot (n_0 - 1)
    \\&\le \frac{n}{k} + \max\{r,s\}\cdot (n_0 - 1).
\end{align*}
This implies
\begin{align*}
    2\bigp{\frac{n-n_0+1}{k}}&\le \frac{n}{k} + \max\{r,s\}\cdot (n_0 - 1).
\end{align*}
Rearranging yields
\begin{align*}
    \frac{n}{k}\le \max\{r,s\}\cdot (n_0 - 1) + \frac{2(n_0-1)}{k} &= \frac{(k\cdot \max\{r,s\} + 2)(n_0-1)}{k},
\end{align*}
a contradiction to $$n> (k\cdot \max\{r,s\} + 2)(n_0-1).$$

Since $n_0$ is fixed but can be arbitrarily large, we can thus build a sequence $\{n_0,n_1,\ldots\}$ using the above procedure\footnote{Let $n_1$ be the number $n$ produced by the procedure. Then we repeat the procedure with $n_1$ in place of $n_0$ to produce $n_2,$ and so on.} such that in each of the intervals $[n_j,n_{j+1}]$ for $j\ge 0$ there is a zero-sum $k$-block. Hence there are infinitely many zero-sum $k$-blocks.
\end{proof}



\section{Avoidance of zero-sum arithmetic subsequences in $\{-1,1\}$-sequences}\label{arithmetic}

In this section and the next we derive lower bounds for the constants $M(r,s,k)$ in Conjecture \ref{quadratic-conjecture} depending on $r$, $s$, and $k.$ The constant $M(r,s,k)$ exists because bounds obtained for the case of zero-sum $k$-blocks serve as upper bounds for $M(r,s,k).$ In other words, we have $$M(r,s,k)\le N(r,s,k).$$ We illustrate our method by first dealing with the special case $r=s=1$ which is probably of most interest since \cite{CHM} deals exclusively with $\{-1,1\}$-sequences. For number theoretic reasons, we can prove our construction yields a quadratic lower bound for $n$ in terms of $k$ only for $(r,s)=(1,1),(1,2),$ and possibly a finite number of other pairs of positive integers $(r,s).$ Nevertheless, our construction yields, for sufficiently large $k,$ a superlinear bound for arbitrary $(r,s)$ as described in Theorem \ref{not-quadratic-theorem} in Section \ref{superlinear}.\footnote{A linear lower bound for $n$ in terms of $k$ is trivial since a sequence of length $n\le k-1$ does not even contain a subsequence of length $k.$}

In order to prove a quadratic lower bound on $M(1,1,k),$ we need to construct a zero-sum sequence of $-1$'s and $1$'s that does not have a zero-sum subsequence indexed by a $k$-term arithmetic progression. One idea to construct such a function $f:[n]\to \{-1,1\}$ is to let the sign of $f(j)$ depend only on the residue of $j$ modulo $k.$ The reason is that this choice results in a nice structure of $k$-term arithmetic progressions when we consider the multiset of residues of the $k$ terms modulo $k$; see Proposition \ref{proposition-residue}. However, it turns out that this construction only yields a quadratic bound for $k\equiv 2\bmod 4$ and so is relegated to the Appendix. In this section we obtain both a better bound and one that holds for all even $k.$ In this construction, the sign of $f(j)$ depends on the residue of $j$ mod $k+1$ rather than $k.$ 

We begin our study on zero-sum arithmetic progressions in $\{-1,1\}$-sequences with a fairly obvious proposition.

\begin{proposition}\label{proposition-residue}
Let $A$ be a $k$-term arithmetic progression of integers with common difference $d.$ Consider the multiset $S$ of the residues of $A$ modulo $k.$ Then the following properties are true:
\begin{enumerate}
    \item The distinct elements of the multiset form an arithmetic progression $A'\subeq [0, k-1]$ with $\frac{k}{\gcd(d,k)}$ terms and common difference $\gcd(d,k).$
    \item The multiplicity of every element of $S$ is $\gcd(d,k).$
\end{enumerate}
\end{proposition}

\begin{proof}
First consider the special case when $\gcd(d,k)=1.$ It is an elementary fact that since $d$ and $k$ are relatively prime then any $k$-term arithmetic progression with common difference $d$ will cover all $k$ residues modulo $k.$ This implies both of the properties above.

In general, suppose that an arithmetic progression of integers $A$ satisfies both of the properties above. We can see that adding a fixed constant $c$ to every term of $A$ preserves both properties. In particular, each element $s$ of the multiset $S$ will be translated by $c\bmod k.$ Hence by translating $A$ it suffices to prove the properties when $0$ is the first term of $A.$ In this case every element of $A$ is divisible by $d$ and hence also $\gcd(d,k),$ so the only possible residues modulo $k$ are those that are divisible by $\gcd(d,k).$ Consider the first $\frac{k}{\gcd(d,k)}$ terms of $A.$ Because $\gcd\bigp{d, \frac{k}{\gcd(d,k)}}  = 1,$ a scaled argument of the special case above implies that the first $\frac{k}{\gcd(d,k)}$ terms of $A$ cover each residue modulo $k$ divisible by $\gcd(d,k),$ and this set of residues is indeed an arithmetic progression $A'\subeq [0, k-1]$ with $\frac{k}{\gcd(d,k)}$ terms and common difference $\gcd(d,k).$ Because $k$ divides $ \frac{k}{\gcd(d,k)}\cdot d,$ we see that in fact every successive block of $\frac{k}{\gcd(d,k)}$ terms of $A$ will have the same residues modulo $k.$ Because there are $\gcd(d,k)$ blocks in total, we conclude that the multiplicity of every residue that appears is $\gcd(d,k).$
\end{proof}

\begin{lemma}\label{plus-lemma}
Let $a=k+1\ge 3$ be odd and let $f:[0, a-1]\to \{-1,1\}$ be a function mapping residues modulo $a$ to $-1$ or $1$ such that
$$f(j) = \begin{cases} -1 & j\bmod a < \frac{a-3}{2} \\ 1 & j\bmod a \ge \frac{a - 3}{2}.\end{cases}$$ Then the following properties hold:
\begin{enumerate}
    \item The number of $f(j)=1$ is $\frac{a+3}{2}$ and the number of $f(j)=-1$ is $\frac{a-3}{2}.$
    \item Let $A\subeq \ZZ$ be an $a$-term arithmetic progression and let $S$ denote the multiset of residues of $A$ modulo $a.$ Then $$\abs{\sum_{j\in S} f(j)} \ge 3,$$ where the sum is over all $\abs{A}$ residues of the multiset $S,$ not just distinct ones.
    \item Let $B\subeq \ZZ$ be a $k$-term arithmetic progression and let $T$ denote the multiset of residues of $B$ modulo $a.$ Then $$\sum_{j\in T} f(j)\neq 0,$$ where the sum is over all $\abs{B}$ residues of the multiset $T,$ not just distinct ones.
\end{enumerate}
\end{lemma}

\begin{proof}
\begin{enumerate}
    \item This is by definition.
    \item By Proposition \ref{proposition-residue} the distinct elements of $S$ form an arithmetic progression $A'\subeq [0, a-1]$ with common difference $d.$ If $\gcd(d,a)=1$ then $A' = [0, a-1],$ and we are done by (1). Otherwise, since $a$ is odd we have $\gcd(d,a)\ge 3.$ Since Proposition \ref{proposition-residue} tells us that $\gcd(d,a)$ is the multiplicity of every element of $S,$ this multiplicity is also at least 3. We conclude that the sum $$\sum_{j\in S} f(j)$$ is divisible by an odd number that is at least 3. Since $a$ is odd this sum cannot be 0, implying the result.
    
    \item Extend $B$ by one term to create an $a$-term arithmetic progression $A$ in $\ZZ.$ Let $S$ denote the multiset of residues of $A$ modulo $a.$ By (2) we have $$\abs{\sum_{j\in S} f(j)} \ge 3.$$ But the multiset $T$ has exactly one more element than $S.$ Since $\abs{f(j)}= 1$ for all $j,$ we conclude that
    \begin{align*}
        \abs{\sum_{j\in S} f(j) - \sum_{j\in T} f(j)} &\le 1 \\
        \abs{\sum_{j\in T} f(j)} &\ge 1.
    \end{align*}
\end{enumerate}
\end{proof}

We are now ready to resolve Conjecture \ref{quadratic-conjecture} for all even $k$ by proving Theorem \ref{quadratic-theorem} from Section \ref{introduction}, which recall states that $$\frac{1}{6}k^2 + O(k) \le M(1,1,k) \le \frac{1}{4}k^2 + O(k)$$ for all even $k.$

\begin{proof}[Proof of Theorem \ref{quadratic-theorem}]
Let $a=k+1$ be odd. Construct a function $f:[0, n-1]\to \{-1,1\}$ for $$n = (a+3)\cdot \floor{\frac{a-3}{6}}$$ by defining $$f(j) = \begin{cases} -1 & j\bmod a < \frac{a-3}{2} \\ 1 & j\bmod a \ge \frac{a-3}{2}.\end{cases}$$

We first verify that the constructed sequence is zero-sum. In each of the first $\floor{\frac{a-3}{6}}$ consecutive blocks of $a$ elements of the sequence there are 3 more $1$'s than $-1$'s, so in total there are $3\floor{\frac{a-3}{6}}$ more $1$'s than $-1$'s in the first $a\cdot \floor{\frac{a-3}{6}}$ terms of the sequence. There are $3\floor{\frac{a-3}{6}}$ remaining terms that are all equal to $-1$ by construction. We conclude that $f([0, n-1]) = 0.$

The fact that any $k$-term arithmetic progression $A\subeq [0, n-1]$ satisfies $f(A)\neq 0$ is a corollary of (3) of Lemma \ref{plus-lemma}.

By our construction, we thus conclude that
\begin{align*}
    M(1,1,k) &\ge (a+3)\cdot \floor{\frac{a-3}{6}}
    \\&= (k+4)\cdot \floor{\frac{k-2}{6}}
    \\&= \frac{1}{6}k^2 +O(k)
\end{align*}
as desired.
\end{proof}

\section{Avoidance of zero-sum arithmetic subsequences in $\{-r,s\}$-sequences}\label{superlinear}

In this section we first generalize the results on zero-sum arithmetic subsequences from $\{-1,1\}$-sequences to general $\{-r,s\}$-sequences. The construction here will use the same idea as the proof of Theorem \ref{quadratic-theorem} of letting $f(j)$ depend on the residue of $j$ modulo some positive integer. However, for arbitrary positive integers $r$ and $s$ we cannot prove a quadratic lower bound in terms of $k$ for $M(r,s,k)$ due to some number-theoretic obstructions.

Recall that we can assume without loss of generality that $\gcd(r,s) = 1$ and that $r \neq s$ since the case $r=s=1$ has already been dealt with in the previous section. We can also assume that $r < s$ since we can negate every term of a $\{-r,s\}$-sequence to produce a $\{-s,r\}$-sequence such that the existence of a zero-sum arithmetic subsequence is preserved. We begin the construction with some definitions and a key lemma comparable to Lemma \ref{plus-lemma}.

\begin{definition}
For a nonnegative integer $\alpha,$ define the set $$\cS_\alpha(r,s) = \{-r\alpha + (r+s)t\mid t\in [0,\alpha]\}$$ to be the set of all possible weights of a $\{-r,s\}$-sequence with length exactly $\alpha.$ Note that $\cS_\alpha(r,s)$ is an arithmetic progression with $\alpha+1$ terms and common difference $r+s.$ 
\end{definition}

\begin{definition}\label{shift-definition}
Fixing an integer $k$ divisible by $r+s,$ we say that $\alpha\ge 0$ is a \emph{good shift} for the tuple $(r,s,k)$ if $k+\alpha$ is relatively prime to every element of $\cS_{\alpha}(r,s).$
\end{definition}

\begin{example}\label{11-example}
In the context of $\{-1,1\}$-sequences, namely $(r,s) = (1,1),$ we can compute $\cS_1(1,1) = \{-1,1\}.$ It is clear that for every $k\ge 1,$ $k+1$ is relatively prime to $-1$ and $1.$ This implies that $\alpha= 1$ is a good shift for all $(1,1,k).$
\end{example}

\begin{example}\label{12-example}
Consider the case $(r, s) = (1,2).$ We compute $\cS_1(1,2) = \{-1,2\}$ and $\cS_2(1,2) = \{-2, 1, 4\}.$ We claim that $\alpha=1$ or $\alpha = 2$ is a good shift for each $(1,2,k).$ The only prime that divides any element of $\cS_1(1,2)$ or $\cS_2(1,2)$ is 2. Hence choosing $\alpha\in \{1,2\}$ such that $\alpha\equiv k+1\bmod 2$ yields a good shift for $(1,2,k).$
\end{example}

We will now see that finding a small good shift leads to a construction that implies a lower bound for $M(r,s,k).$ We first introduce a key lemma.

\begin{lemma}\label{shift-lemma}
Let $\alpha\ge 0$ be a good shift for $(r,s,k)$ and let $a=k+\alpha.$ Let $f:[0, a-1]\to \{-r,s\}$ be a function mapping residues modulo $a$ to $-r$ or $s$ such that $$f(j) = \begin{cases} -r & j\bmod a < \frac{sk}{r+s} - 1 \\ s & j\bmod a \ge \frac{sk}{r+s} - 1.\end{cases}$$ Then the following properties hold:
\begin{enumerate}
    \item The number of $-r$'s is $\frac{sk}{r+s} - 1$ and the number of $s$'s is $\frac{rk}{r+s} + 1 + \alpha.$
    
    \item Let $A\subeq \ZZ$ be an $a$-term arithmetic progression and let $S$ denote the multiset of residues of $A$ modulo $a.$ Then $$\sum_{j\in S} f(j)\notin \cS_\alpha(r,s),$$ where the sum is over all $\abs{A}$ residues of the multiset $S,$ not just distinct ones.
    
    \item Let $B\subeq \ZZ$ be a $k$-term arithmetic progression and let $T$ denote the multiset of residues of $B$ modulo $a.$ Then $$\sum_{j\in T}f(j)\neq 0,$$ where the sum is over all $\abs{B}$ residues of the multiset $T,$ not just distinct ones.
\end{enumerate}
\end{lemma}

\begin{proof}
\begin{enumerate}
    \item This is by definition.
    
    \item By Proposition \ref{proposition-residue} the distinct elements of $S$ form an arithmetic progression $A'\subeq [0, a-1],$ say with common difference $d.$
    
    If $\gcd(d,a)=1$ then $A' = [0, a-1].$ We note that
    \begin{align*}
        \sum_{j=0}^{a-1} f(j) &= -r\bigp{\frac{sk}{r+s}-1} + s\bigp{\frac{rk}{r+s} + 1 + \alpha}
        \\&= r + s + s\alpha,
    \end{align*}
    which is strictly greater than any element in $\cS_\alpha(r,s).$ Hence we are done.
    
    Otherwise we have $\gcd(d,a)\ge 2.$ Proposition \ref{proposition-residue} tells us that $\gcd(d,a)$ is the multiplicity of every element of $S.$ We conclude that the sum $$\sum_{j\in S} f(j)$$ is divisible by $\gcd(d,a)$ and hence by a prime number that divides $a.$ But since $\alpha$ is assumed to be a good shift for $(r,s,k),$ by definition $a$ is relatively prime to every element of $\cS_\alpha(r,s),$ hence $$\sum_{j\in S} f(j)\notin \cS_\alpha(r,s).$$
    
    \item Extend $B$ by $\alpha$ terms to create an $a$-term arithmetic progression $A$ in $\ZZ.$ Let $S$ denote the multiset of residues of $A$ modulo $a.$ The multiset $T$ has exactly $\alpha$ more elements than $S.$ Hence
    \begin{align*}
        \sum_{j\in S} f(j) - \sum_{j\in T} f(j) &\in \cS_\alpha(r,s).
    \end{align*}
    If we had $$\sum_{j\in T} f(j) = 0,$$ then this would imply that $$\sum_{j\in S}f(j)\in \cS_\alpha(r,s),$$ which is a contradiction because of (2).
\end{enumerate}
\end{proof}

\begin{theorem}\label{good-shift-theorem}
Let $r,$ $s,$ and $k$ be positive integers such that $r<s,$ $\gcd(r,s)=1,$ and $r+s$ divides $k.$ Let $\alpha$ be a good shift for $(r,s,k).$ Then
\begin{align*}
    M(r,s,k) &\ge \bigp{r(k+\alpha) + (r+s+s\alpha)}\cdot \floor{\frac{\frac{sk}{r+s}-1}{r(r+s+s\alpha)}}.
\end{align*}
\end{theorem}

\begin{proof}
Let $a = k+\alpha.$ Construct a function $f:[0, n-1]\to \{-r,s\}$ for $$n= \bigp{ra + (r+s+s\alpha)}\cdot \floor{\frac{\frac{sk}{r+s}-1}{r(r+s+s\alpha)}}$$ by defining $$f(j) = \begin{cases} -r & j\bmod a < \frac{sk}{r+s} - 1 \\ s & j\bmod a \ge \frac{sk}{r+s} - 1.\end{cases}$$

We first verify that the constructed sequence is zero-sum. Each of the first $$r\cdot \floor{\frac{\frac{sk}{r+s}-1}{r(r+s+s\alpha)}}$$ consecutive blocks of $a$ elements of the sequence has weight $r+s+s\alpha$ by a calculation in (2) of Lemma \ref{shift-lemma} noting that $d=1$ in our case. There are $$(r+s+s\alpha)\cdot \floor{\frac{\frac{sk}{r+s}-1}{r(r+s+s\alpha)}} \le \frac{sk}{r+s}-1$$ remaining terms that are all equal to $-r$ by Definition \ref{shift-definition}. Hence the total weight is $$r\cdot \floor{\frac{\frac{sk}{r+s}-1}{r(r+s+s\alpha)}}\cdot (r+s+s\alpha) - r(r+s+s\alpha)\cdot \floor{\frac{\frac{sk}{r+s}-1}{r(r+s+s\alpha)}} = 0.$$ We conclude that $f([0,n-1]) = 0.$

The fact that any $k$-term arithmetic progression $A\subeq [0,n-1]$ satisfies $f(A)\neq 0$ is a corollary of (3) in Lemma \ref{shift-lemma}.

By our construction, we conclude that $$M(r,s,k) \ge \bigp{r(k+\alpha) + (r+s+s\alpha)}\cdot \floor{\frac{\frac{sk}{r+s}-1}{r(r+s+s\alpha)}}.$$
\end{proof}

Recall that our construction in the proof of Theorem \ref{good-shift-theorem} relies on finding a good shift for $(r,s,k).$ The question of finding the smallest good shift $\alpha$ for $(r,s,k)$ now remains, which is a number theoretic question. If one could prove that for any fixed positive integers $r$ and $s$ there is a uniform upper bound on $\alpha$ for all $k$ then Theorem \ref{good-shift-theorem} would imply a quadratic lower bound for $M(r,s,k).$

\begin{remark}
Note that this construction is in fact a generalization of our construction in the proof of Theorem \ref{quadratic-theorem} for $\{-1,1\}$-sequences. Indeed, Example \ref{11-example} confirms that $\alpha=1$ is a good shift for all $(1,1,k).$
\end{remark}

Note that Example \ref{12-example} guarantees a good shift $\alpha$ for each $(1,2,k)$ that is at most 2. Theorem \ref{good-shift-theorem} thus implies the quadratic lower bound
\begin{align*}
    M(r,s,k) &\ge \bigp{r(k+\alpha) + (r+s+s\alpha)}\cdot \floor{\frac{\frac{sk}{r+s}-1}{r(r+s+s\alpha)}}
    \\&\ge \frac{2}{21}k^2 + O(k).
\end{align*}
for $(r,s) = (1,2).$ Combining this result with Theorem \ref{rs-theorem} for $(r,s) = (1,2)$ yields the following corollary:

\begin{corollary}
We have $$\frac{2}{21}k^2 + O(k) \le M(1,2,k) \le \frac{2}{9}k^2 + O(k)$$ for all $k\equiv 0\bmod 3.$
\end{corollary}

Unfortunately we do not know in general for fixed $r$ and $s$ whether or not a uniform upper bound on the minimum good shift $\alpha$ for all $(r,s,k)$ exists. However, using a well-known result by Baker, Harman, and Pintz \cite{BHP} that for sufficiently large $k$ there is a prime in the interval $(k, k+k^{0.525}],$ we can at least say that $M(r,s,k) = \omega(k),$ that is, $M(r,s,k)$ is superlinear in $k.$ We now give the proof of Theorem \ref{not-quadratic-theorem} from Section \ref{introduction}, which recall states that for positive integers $r,$ $s,$ and $k,$ such that $r+s$ divides $k,$ we have $$\Omega_{r,s}\bigp{k^{1.475}}\le M(r,s,k)\le \frac{rs}{(r+s)^2}k^2 + O(k)$$ for sufficiently large $k.$

\begin{proof}[Proof of Theorem \ref{not-quadratic-theorem}]
Fix arbitrary positive integers $r$ and $s$ such that $r<s$ and $\gcd(r,s) = 1.$ For sufficiently large $k,$ choose $\alpha \le k^{0.525}$ such that $k+\alpha$ is prime and such that $k + \alpha > s\alpha.$ The only prime factor of $k+\alpha$ is itself, which is greater than any element in $S_\alpha.$ We conclude that $\alpha$ is a good shift for $(r,s,k).$ The bound now becomes
\begin{align}\label{superlinear-equation}
    M(r,s,k) &\ge \bigp{r(k+\alpha) + (r+s+s\alpha)}\cdot \floor{\frac{\frac{sk}{r+s}-1}{r(r+s+s\alpha)}}
    \\&= \Omega_{r,s}\bigp{k\cdot \frac{k}{\alpha}}
    \\&= \Omega_{r,s}\bigp{k^{1.475}}
\end{align}
for sufficiently large $k.$

In the case that $r > s,$ a negation of the sequence followed by the same argument above shows that $M(r,s,k) = \Omega_{r,s}\bigp{k^{1.475}}$ as well.
\end{proof}

\begin{remark}
The statement $M(r,s,k) = \Omega_{r,s}\bigp{k^{1.475}}$ for sufficiently large $k$ is equivalent to the statement that there exists a constant $c_{r,s}>0$ such that $M(r,s,k) \ge c_{r,s}\cdot k^{1.475}$ for sufficiently large $k.$ We remark that the constant $c_{r,s}$ is effective and only depends on $r$ and $s.$ In particular, from \eqref{superlinear-equation} we have $$c_{r,s}\approx \frac{1}{r+s}.$$
\end{remark}

We now turn our focus to upper bounds on $M(r,s,k).$ Recall that we have the trivial bound $M(r,s,k)\le N(r,s,k)$ because a $k$-block is an arithmetic progression of length $k$ and common difference $1.$ A natural question is whether the upper bounds for $M(r,s,k)$ can be improved. It would be particularly interesting if one can prove $$M(r,s,k)\le c_2 k^2 + O(k)$$ for a constant $c_2 < \frac{rs}{(r+s)^2}$ for certain pairs $(r,s).$ Results such as these would imply that forbidding $k$-term arithmetic progressions decreases the number of terms in a zero-sum sequence by at least a constant ratio in these cases. That being said, at least in the case of $r=s=1$ the constant $c_2 = \frac{1}{4}$ in the upper bound $$M(1,1,k) \le \frac{1}{4}k^2 + O(k)$$ cannot be improved.

\begin{theorem}\label{11-tight}
There are infinitely many even $k$ for which $$M(1,1,k)\ge \frac{1}{4}k^2 + O(k).$$
\end{theorem}

\begin{proof}
We show by construction that $$M(1,1,k)\ge \frac{1}{4}k^2 -1$$ for all $k=2p,$ where $p$ is an odd prime. Construct a function $f:[0, n-1]\to \{-1,1\}$ for $$n = (2p+2)\bigp{\frac{p-1}{2}} = p^2 - 1.$$ by defining $$f(j) = \begin{cases} -1 & j\bmod k < p-1 \\ 1 & j\bmod k \ge p-1.\end{cases}$$

We see that the constructed sequence is zero-sum because in each of the first $\frac{p-1}{2}$ consecutive blocks of $k=2p$ elements there are two more $1$'s than $-1$'s, and the remainder block has $p-1$ $-1$'s, yielding a total weight of $$2\cdot \frac{p-1}{2}  - (p-1) = 0.$$

We now prove that any $k$-term arithmetic progression $A\subeq [0, n-1]$ with common difference $d$ satisfies $f(A)\neq 0.$ Restricting $f$ to a function $\ZZ/k\ZZ\to \{-1,1\},$ it suffices to check that $f(A')\neq 0$ for any arithmetic progression $A'\subeq [0, k-1]$ with $\frac{k}{\gcd(d,k)}$ terms and common difference $\gcd(d,k).$

We check all of the possible common differences $d$:
\begin{itemize}
    \item The case $d=1$ is handled by Lemma \ref{rs-construction}, which tells us that there is no zero-sum $k$-block.
    \item The case $d=p$ is impossible since in order for a $k$-term arithmetic subsequence to exist we would have to have $$n > (k-1)d = 2p^2 - p$$ terms, which is greater than $n = p^2 - 1.$
    \item Otherwise the number of terms in $A'$ is not divisible by 2, meaning that it is impossible for $A'$ to be zero-sum.
\end{itemize}
\end{proof}

\section{Open problems}\label{open}

We end with a summary of open problems that arise from our work.

\begin{problem}
A natural open question is whether the lower bounds for $M(r,s,k)$ can be improved. It is very unlikely that $\Omega_{r,s}(k^{1.475})$ is the best lower bound that one can prove, and recall that Caro et al.\! guess in Conjecture \ref{quadratic-conjecture} that the lower bound in terms of $k$ should be quadratic.
\end{problem}

\begin{problem}
Recall that Theorem \ref{11-tight} shows that the constant $c_2(r,s)$ is tight for $(r,s) = (1,1).$ Using the same argument it is not hard to show tightness for any $(r,s)$ such that $r+s > rs,$ namely $r=1$ or $s=1.$ But one can wonder whether $$M(r,s,k)\le c_2 k^2 + O(k)$$ for a constant $c_2 < \frac{rs}{(r+s)^2}$ for certain other pairs of $(r,s).$
\end{problem}

\begin{problem}
Determine, for all fixed positive integers $r$ and $s,$ whether or not there exists a constant $\ol{\alpha}(r,s)$ such that for each $k$ we can find a good shift $\alpha\le \ol{\alpha}(r,s)$ for $(r,s,k).$\footnote{Recall the definition of a good shift from Definition \ref{shift-definition}.} Note that this would then imply a quadratic lower bound for $M(r,s,k)$ for each pair of $r$ and $s.$
\end{problem}

\appendix
\section{}
The purpose of this section is to discuss the more natural construction of a $\{-1,1\}$-sequence that avoids a zero-sum $k$-term arithmetic subsequence in which the terms of the sequence only depend on the index of the term modulo $k.$ As mentioned in Section \ref{arithmetic}, this construction achieves a quadratic asymptotic bound that is suboptimal to the construction in Section \ref{arithmetic}.

The idea is to assign to each residue modulo $k$ the value $-1$ or $1$ such that there is no zero-sum arithmetic progression of the form $A'$ as described in Proposition \ref{proposition-residue}. As we shall see, to get the best lower bound on $M(1,1,k)$ one should always set $m=1$ in Lemma \ref{lemma-residue} below. Nevertheless, because the construction in the following lemma generalizes nicely to any positive integer $m$ and may be of independent interest, we describe it:

\begin{lemma}\label{lemma-residue}
Suppose that $k=2 a_1 a_2\cdots a_m$ for some $m\ge 1$ and integers $a_1,a_2,\ldots,a_m>1,$ and furthermore assume that $2,a_1,a_2,\ldots,a_m$ are pairwise relatively prime. Let $f:[0, k-1]\to \{-1,1\}$ be a function mapping residues modulo $k$ to $-1$ or $1$ with $f(j)$ computed as follows:
\begin{enumerate}
    \item Define $x\bmod y$ to be the integer $x'\in [0, y-1]$ for which $x\equiv x'\bmod y.$ For each $0\le i\le m,$ define $$r_i(j) = \begin{cases} -1 & j\bmod a_i < \frac{a_i-1}{2} \\ 1 & j\bmod a_i \ge \frac{a_i-1}{2}.\end{cases}$$
    \item Set $$f(j) = \prod_{1\le i\le m} r_i(j).$$
\end{enumerate}
Then the following properties hold:
\begin{enumerate}
    \item The number of $f(j)=1$ is $\frac{k}{2}+1$ and the number of $f(j)=-1$ is $\frac{k}{2} - 1.$
    \item Let $d$ be any positive integer dividing $k.$ For any arithmetic progression $A\subeq [0, k-1]$ with $\frac{k}{d}$ terms and common difference $d,$ we have $f(A)\neq 0.$
\end{enumerate}
\end{lemma}

\begin{proof}
\begin{enumerate}
    \item By the Chinese Remainder Theorem we have
    \begin{align*}
        \sum_{j\in \ZZ/k\ZZ} f(j)
        &= \sum_{j\in (\ZZ/2\ZZ) \times \prod_{1\le i\le m} (\ZZ/a_i \ZZ)} \bigp{\prod_{1\le i\le m} r_i(j)}
        \\&= \bigp{\sum_{j\in \ZZ/2\ZZ} 1} \bigb{\prod_{1\le i\le m} \bigp{\sum_{j\in \ZZ/a_i \ZZ} r_i(j)}}
        \\&= 2 \prod_{1\le i\le m} \bigp{\frac{a_i + 1}{2} - \frac{a_i - 1}{2}}
        \\&= 2.
    \end{align*}
    Hence the number of $1$'s minus the number of $-1$'s is 2. Combined with the fact that the sum of the number of $1$'s and the number of $-1$'s is $k,$ we deduce that the number of $1$'s is $\frac{k}{2}+1$ and the number of $-1$'s is $\frac{k}{2} - 1.$
    
    
    \item If $d$ is even then $\frac{k}{d}$ is odd. Since no $\frac{k}{d}$ numbers each of which is in $\{-1,1\}$ can sum to 0 we are done. Hence we can assume that $d$ is odd, meaning that half the residues in $A$ are even and half are odd. By the Chinese Remainder Theorem, every residue modulo $a_1a_2\cdots a_m$ corresponds to two residues modulo $k=2a_1a_2\cdots a_m.$ Because $f(j)$ does not depend on $j\bmod 2,$ we have
    \begin{equation}\label{fA}
        f(A) = 2\sum_{\substack{j\in A\\ j\equiv 0\bmod 2}} f(j).
    \end{equation}
    Note that there are an odd number of such $j$ satisfying $j\in A$ and $j\equiv 0\bmod 2$ since $k\equiv 2\bmod 4.$ We conclude that the sum in \eqref{fA} is nonzero, implying $f(A)\neq 0$ as well.
    
\end{enumerate}
\end{proof}

In the context of achieving the best lower bound on $M(1,1,k),$ it is always optimal to set $m=1$ in Lemma \ref{lemma-residue}.

\begin{corollary}\label{corollary-residue}
Suppose that $k$ is of the form $2a$ for some $a>1$ that is odd. Let $f:[0, k-1]\to \{-1,1\}$ be a function mapping residues modulo $k$ to $-1$ or $1$ such that $$f(j) = \begin{cases} 1 & j\bmod a < \frac{a-1}{2} \\ -1 & j\bmod a \ge \frac{a-1}{2}.\end{cases}$$
Then the following properties hold:
\begin{enumerate}
    \item The number of $f(j)=1$ is $\frac{k}{2}+1$ and the number of $f(j)=-1$ is $\frac{k}{2} - 1.$
    \item Let $d$ be any positive integer dividing $k.$ For any arithmetic progression $A\subeq [0, k-1]$ with $\frac{k}{d}$ terms and common difference $d,$ we have $f(A)\neq 0.$
    \item We have $f(j) = -1$ for $j=0,1,\ldots,\frac{a-3}{2},$ which is a total of $\frac{a-1}{2}$ numbers.
\end{enumerate}
\end{corollary}

\begin{proof}
Properties (1) and (2) follow from Lemma \ref{lemma-residue}. Property (3) can be checked by definition of the construction.
\end{proof}

We now construct for $k\equiv 2\bmod 4$ a function $f:[n]\to \{-1,1\}$ that has no zero-sum subsequence indexed by a $k$-term arithmetic progression for $n$ growing quadratically in $k.$ For notational simplicity we construct a function $f:[0, n-1]\to \{-1,1\}$ instead, which is equivalent.

\begin{theorem}
We have $$M(1,1,k)\ge \frac{1}{8}k^2 + O(k)$$ for $k\equiv 2\bmod 4.$
\end{theorem}
\begin{proof}
Let $a>1$ be odd. Construct a function $f:[0, n-1]\to \{-1,1\}$ for $$n = (2a+2)\cdot \floor{\frac{a-1}{4}}$$ by defining $$f(j) = \begin{cases} -1 & j\bmod a < \frac{a-1}{2} \\ 1 & j\bmod a \ge \frac{a-1}{2}.\end{cases}$$

We first verify that our constructed sequence is zero-sum. By Corollary \ref{corollary-residue}, in each of the first $\floor{\frac{a-1}{4}}$ consecutive blocks of $k$ elements of the sequence, there are two more $1$'s than $-1$'s, so in total there are $2\floor{\frac{a-1}{4}}$ more $1$'s than $-1$'s in the first $2a\cdot \floor{\frac{a-1}{4}}$ terms of the sequence. By (3) of Corollary \ref{corollary-residue}, for $j\ge 2a\cdot \floor{\frac{a-1}{4}}$ we have $f(j) = -1.$ In other words, the remaining $2\cdot \floor{\frac{a-1}{4}}$ are all $-1$'s. We conclude that $f([0,n-1]) = 0.$

Now we verify that our constructed sequence has no zero-sum arithmetic progression subsequence.

\begin{claim}
Let $A\subeq [0, n-1]$ be any $k$-term arithmetic progression with common difference $d$. Then $f(A)\neq 0,$ and furthermore $\abs{f(A)} \ge \gcd(d,k).$
\end{claim}
To prove this claim, consider the multiset $S$ of the residues of $A$ modulo $k.$ By Proposition \ref{proposition-residue}, the distinct elements of $S$ form an arithmetic progression $A'\subeq [0, k-1]$ with $\frac{k}{\gcd(d,k)}$ terms and common difference $\gcd(d,k).$ Furthermore the multiplicity of every element of $S$ is $\gcd(d,k).$ Note that $f(j)$ depends only on the residue of $j$ modulo $k.$ Hence restricting $f$ to a function from $\ZZ/k\ZZ$ to $\{-1,1\},$ it suffices to show that $f(A')\neq 0.$ But this is the result of Lemma \ref{lemma-residue}. This completes the proof of the claim.

By our construction, we thus conclude that
\begin{align*}
    M(1,1,k) &\ge (2a+2)\cdot \floor{\frac{a-1}{4}}
    = (k+2)\cdot \floor{\frac{k-2}{8}},
\end{align*}
which implies $$M(1,1,k)\ge \frac{1}{8}k^2 + O(k)$$ in the case $k\equiv 2\bmod 4.$
\end{proof}

It is unlikely that any construction letting $f(j)$ depend on only the residue $j\bmod k$ according to Lemma \ref{lemma-residue} produces a quadratic lower bound for $n$ when $k$ is divisible by a large power of 2. The following proposition, which may be of independent interest, presents a natural barrier to using such a construction for all even $k$:

\begin{proposition}
Let $k=2^v$ be a power of 2 and let $f:[0, k-1]\to \{-1,1\}$ be any function. Suppose for any $2^{v-v'}$-term arithmetic progression $A$ with common difference $2^{v'}$ in $[0, k-1]$ we have $f(A)\neq 0.$ Then either $f(j) = -1$ for all $j$ or $f(j) = 1$ for all $j.$
\end{proposition}
\begin{proof}
We will prove by induction on $v'$ that the image of any $2^{v-v'}$-term arithmetic progression $A$ with common difference $2^{v'}$ in $[0, k-1]$ is either only $-1$ or only $1.$ Note that any such $A$ is of the form $j,j+2^{v'},\ldots,j+2^v-2^{v'}$ for $j\in [0, 2^{v'}-1].$

The induction will start from $v'=v-1$ and work downwards. In the case where $v' = v-1$ we know that since $$f(j) + f(j+2^{v-1}) \neq 0,$$ $f(j)$ and $f(j+2^{v-1})$ must have the same sign.

Now assume the inductive hypothesis is true for $v'+1.$ We want to show it is true for $v'.$ Consider an arithmetic progression $A$ of the form $j,j+2^{v'},\ldots,j+2^v-2^{v'}$ for $j\in [0, 2^{v'}-1].$ We note that $A$ is the disjoint union of two arithmetic subsequences $$A_1 = j,j+2^{v'+1},\ldots,j+2^v-2^{v'+1}$$ and $$A_2 = j+2^{v'}, j+2^{v'}+2^{v'+1},\ldots,j+2^{v'}+2^v-2^{v'+1},$$ both of which are of the form sufficient for the inductive hypothesis. Hence the image of $A_1$ and the image of $A_2$ under $f$ are both constant. If the images were different in sign, then since $A_1$ and $A_2$ have the same number of terms, this would imply that $f(A)=0,$ which contradicts the assumption. We conclude that in fact the image of $A$ is either only $-1$ or only $1,$ completing the inductive step.
\end{proof}

\end{document}